\documentclass[11pt,oneside]{amsart}
\newif\ifarXiv
\arXivtrue

\ifarXiv
\usepackage[a4paper]{geometry}
\fi
\usepackage[latin1]{inputenc}
\usepackage{amssymb}
\usepackage{epsfig}
\usepackage{url}
\usepackage{hyperref}
\usepackage{graphicx}

\ifarXiv
\else
\usepackage{linenum}
\fi

\newtheorem{thm}{Theorem}[section]
\newtheorem{cor}[thm]{Corollary}
\newtheorem{lem}[thm]{Lemma}

\newtheorem{pro}[thm]{Proposition}
\newtheorem{defi}[thm]{Definition}

\ifarXiv
\else
\fi  
\begin{document}
\ifarXiv
\title[R. Grunert, W. K\"uhnel, G. Rote: PL Morse theory in low dimensions]
{PL Morse theory in low dimensions}
\else
\title[PL Morse theory in low dimensions]
{PL Morse theory in low dimensions}
\fi
\author[Romain Grunert, Wolfgang K\"uhnel, and G\"unter Rote]{Romain
  Grunert, Wolfgang K\"uhnel, and G\"unter Rote}
%%%%%%%%%%%%%%%%%%%%%%%%%%%%%%%%%%%%%%%%%%%%%%%%%%%%%%%%%%%%%%%%%%%%%%%
%\begin{document}
%Version \today
\thanks{}
\subjclass[2010]{57R70; %   Critical points and critical submanifolds
 % (würde sich als "primary" eigenen; ist nach meiner Kenntnis
 %                                       Standard für Morse-Theorie)
 %                                       57Q15  oder
  57Q99, % PL-topology
%57M50 (?) % Geometric structures on low-dimensional manifolds
%  57-04
  % Explicit machine computation and programs (not the theory of
  % computation or programming)
52B70, % Polyhedral manifolds
68Q17%  	Computational difficulty of problems
}
\keywords{}

\begin{abstract}
{\small We discuss a PL analogue of Morse theory for PL manifolds. 
There are several notions
of regular and critical points. A point is homologically  regular if the
homology does not change when passing through its level, it is 
strongly regular if the function can serve as one coordinate in a chart.
Several criteria for strong regularity are presented.
In particular we show that in low
dimensions $d \leq 4$ a homologically regular point on a PL $d$-manifold
is always strongly regular.
Examples show that this fails to hold in higher dimensions $d \geq 5$.
One of our constructions involves an 8-vertex embedding of the dunce hat into
a polytopal 4-sphere with 8 vertices such that a regular neighborhood 
is Mazur's contractible 4-manifold. 
}
\end{abstract}
\maketitle
\section{Introduction}
What is nowadays called \emph{Morse Theory} after its pioneer Marston Morse
(1892--1977) has two roots:
One from the calculus of variations  \cite{Morse1}, the other
one from the differential topology of manifolds \cite{Morse2}.
%At a stationary (or critical) point of a variational problem the neighborhood
%can be understood by the second variation. For example a geodesic line
%in a Riemannian manifold is a solution of the variation 
%of the arc length or of the energy functional.
%Then the second variation tells us whether geodesics in the neighborhood -- 
%joining the same two points -- are longer or shorter than the given geodesic.
In both cases, the idea is to consider stationary points for the first variation
of smooth functions or functionals. Then the second variation around
such a stationary point describes the behavior in a neighborhood.
In finite-dimensional calculus this can be completely  described by 
the Hessian of the function provided that the Hessian is non-degenerate.
%This leads to an index form in Riemannian geometry involving the
%curvature tensor, in particular the sign of the sectional curvature,
%see \cite{Morse2}.
In the global theory of (finite-dimensional) differential manifolds,
smooth Morse functions can be used for a decomposition of the manifolds
into certain parts. Here the basic observation is that generically
a smooth real function has isolated critical points (that is, points with
a vanishing gradient), and at each critical point the Hessian
matrix is non-degenerate. 
The index of the Hessian is then taken as the \emph{index}
of the critical point.
This leads to
the Morse lemma and the Morse relations, as well as a handle 
decomposition of the manifold \cite{Morse1,Milnor,MC1,MC2}.
Particular cases are height functions of submanifolds of Euclidean spaces.
Almost all height functions are non-degenerate, and for compact manifolds
the average of the number of critical points equals the total absolute
curvature of the submanifold. Consequently, the infimum of the total
absolute curvature coincides with the \emph{Morse number} of a
manifold, which is defined as the minimum possible number of critical points 
of a Morse function
\cite{Kui1}.

\medskip
Already in the early days of Morse theory, this approach was extended
to non-smooth functions on suitable spaces \cite{Morse3,Morse4,Kui1,Kui2}.
One branch of that development led to several possibilities
of a Morse theory for PL manifolds or for polyhedra in general.

\medskip
First of all, it has to be defined what a critical point is supposed to be
since there is no natural substitute for the gradient and the Hessian
of a function.
Instead the typical behavior of such a function at a critical or non-critical
point has to be adapted to the PL situation.
Secondly, it cannot be expected that non-degenerate points are generic
in the same sense as in the smooth case, at least not extrinsically
for submanifolds of Euclidean space:
For example,
a monkey saddle of a height function on a smooth surface in 3-space
can be split by a small perturbation of the direction of the height vector
into two non-degenerate saddle points.
By contrast, a monkey saddle on a PL surface in space is locally stable under
such perturbations \cite{Ba}.
Abstractly, one can split the monkey saddle into an edge with two endpoints
that are ordinary saddle points, see \cite[Fig.~3]{EHZ}. 
%
%\begin{figure}[hbt]
%\centering
%\epsfig{figure=,height=50mm}
%\caption{Polyhedral monkey saddle and its splitting}
%\end{figure}
%
%
Finally, in higher dimensions we have certain topological phenomena
that have no analogue in classical Morse theory like contractible but not 
collapsible polyhedra, homology points that are not homotopy points,
non-PL triangulations and non-triangulable topological manifolds.

Fro{m} an application viewpoint, piecewise linear functions on domains
of high dimensions arise in many fields, for example from
simulation experiments or from measured data. One powerful way to
explore such a function that is defined, say, on a three-dimensional
domain, is by the interactive
visualization of level sets. In this setting, it is interesting to
know the topological changes between level sets, and
 critical points are precisely those points where
such changes
occur.

\smallskip

After an introductory section about
polyhedra and PL manifolds (Section~\ref{sec:Polyhedra}),
we review the definitions of
regular and critical points in a homological sense
in Section~\ref{sec:critical-points}.
In 
Section~\ref{sec:our-critical-points}, we contrast this with
% we define
what we call \emph{strongly regular} points
(Definition~\ref{Morse function}).
In accordance with classical Morse Theory,
we distinguish the points that are not strongly regular
into non-degenerate critical points and
degenerate critical points,
and we define PL Morse
functions as functions that have no degenerate critical points.
Section~\ref{isotopy} briefly discusses the construction of a PL
isotopy between level sets
across strongly regular points.
Section~\ref{sec:with-boundary} extends the treatment to surfaces with boundary.

%\medskip
Another branch of the development was established by Forman's 
\emph{Discrete Morse theory}  \cite{Fo}. Here in a purely combinatorial way
functions are considered that associate certain values to 
faces of various dimensions in a complex. These Morse functions are 
not a priori continuous functions in the ordinary sense. 
However,
as we show in  Section~\ref{sec:Forman},
they can be turned into PL Morse functions in the sense
defined above.
%and  \cite{Gr}.

While in low dimensions up to 4,
  the weaker notion of H-regularity is sufficient to
 guarantee strong regularity
 (Section~\ref{sec:low-dimensions}),
 this is no longer true in higher dimensions.
 Sections~\ref{sec:high-dimensions}
and~\ref{sec:dunce-hat} give various examples of phenomena that
 arise in high dimensions.
Finally, in
Section~\ref{sec:decidable}, we discuss the algorithmic questions that
arise around the concept of strong regularity. In particular, we show
some undecidability results in high dimensions.

\smallskip The results of Sections \ref{sec:our-critical-points},
\ref{isotopy},
\ref{sec:Forman} and \ref{sec:decidable} are based on the Ph.D.\
thesis of R.~Grunert \cite{Gr}. Some preliminary approaches to these
questions were earlier sketched in \cite{Rote}.

\section{Polyhedra and PL manifolds}
\label{sec:Polyhedra}
\begin{defi}
A topological manifold $M$ is called a \emph{PL manifold} if it is
equipped with a covering $(M_i)_{i \in I}$ of \emph{charts} $M_i$
such that all coordinate transformations between two overlapping charts
are piecewise linear homeomorphisms of open parts of Euclidean space.

\medskip
Fro{m} the practical point of view, a compact PL $n$-manifold $M$ can be
interpreted as a finite \emph{polytopal complex} $K$ built up by convex 
$d$-polytopes such that $|K|$
is homeomorphic with $M$ and such that the star of each (relatively open) cell
is piecewise linearly homeomorphic with an open ball in $d$-space.
Since every polytope can be triangulated, any compact PL $d$-manifold
can be triangulated such that the link of every $k$-simplex is a
combinatorial $(d-k-1)$-sphere. Such a simplicial complex is often called
a \emph{combinatorial $d$-manifold}  \cite{Lu1}.

\end {defi}

\begin{figure}[hbt]
\centering
\includegraphics[height=60mm]{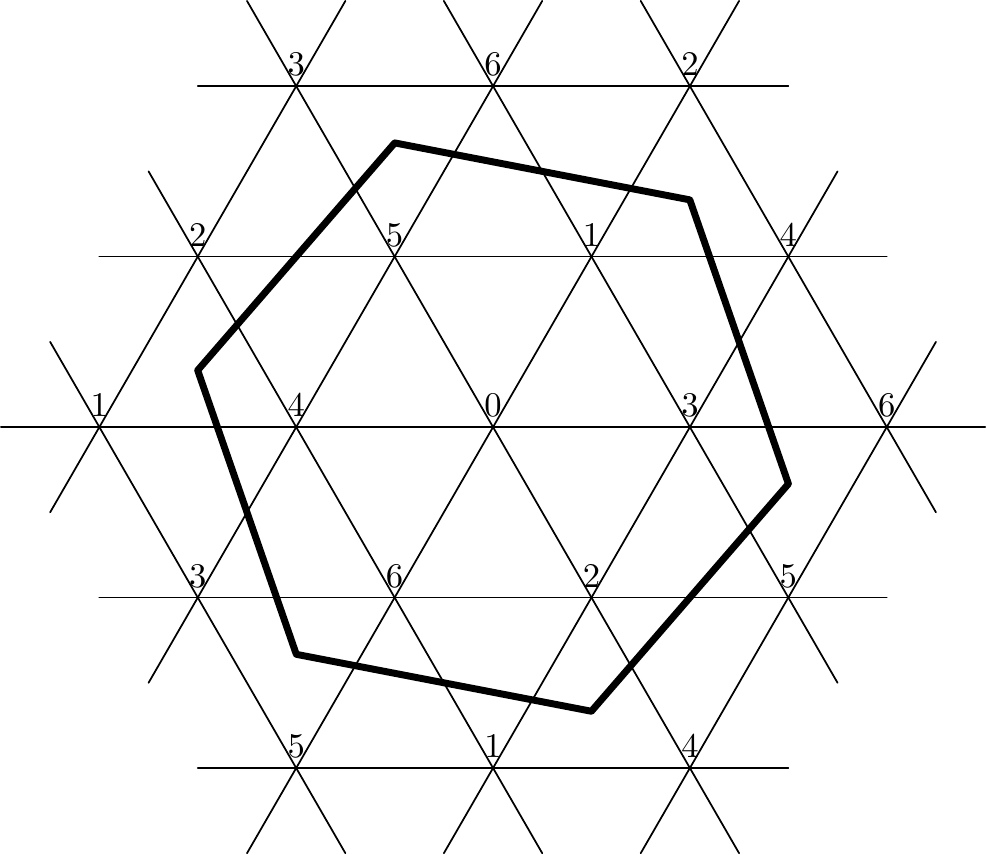}
\caption{The unique 7-vertex triangulation of the torus}
\label{fig:torus}
\end{figure}

In greater generality, one can consider finite polytopal complexes.
In the sequel we will consider a Morse theory for polytopal
complexes in general as well as for combinatorial manifolds.
If the polytopal complex is embedded into Euclidean space
such that every cell is realized by a convex polytope of the
same dimension, then we have the \emph{height functions} defined
as restrictions of linear functions. 

\medskip
A particular case is the
abstract 7-vertex triangulation of the torus (see Figure~\ref{fig:torus}) and its
realization in 3-space \cite{Lu2}.
Observe that a generic PL function with
$f(1) < f(2) < f(4) <  f(0) < \cdots$ has a monkey saddle at 
the vertex 0 since in the link of 0 the sublevel consists
of the three isolated vertices 1, 2, 4. Therefore, passing through the
level of 0 from below will attach two 1-handles simultaneously to 
a disc around the triangle 124. Compare Fig. 11 in \cite[p.99]{Ku1}.

\medskip
For a general outline and the terminology of PL topology we refer to \cite{RS},
where -- in particular -- Chapter 3 introduces the notion of
a \emph{regular neighborhood} of a subpolyhedron of a polyhedron.

\medskip
Occasionally, results in PL topology depend on the Hauptvermutung or
the Schoenflies Conjecture.

\medskip
{\sc The Hauptvermutung:} This conjecture stated that two PL manifolds
that are homeomorphic to one another are also PL homeomorphic to one another.

This conjecture is true for dimensions $d \leq 3$ but systematically false
in higher dimensions. However, it holds for $d$-spheres with $d \not = 4$
and for other special manifolds, compare \cite{Rud}.
 
\medskip
{\sc The PL Schoenflies Conjecture:}
This states the following: A combinatorial $(d-1)$-sphere embedded into
a combinatorial $d$-sphere decomposes the latter into two combinatorial
$d$-balls.

The PL Schoenflies Conjecture is true for $d \leq 3$ and unknown in 
higher dimensions. If however the closure of each component of $S^d 
\setminus S^{d-1}$ is a manifold with boundary, then the 
conclusion of the Schoenflies Conjecture 
is true for all $d \not = 4$ \cite[Ch.3]{RS}.

\section{Regular and critical points of PL functions}
\label{sec:critical-points}

The simplest way to carry over the ideas of Morse theory to PL is
to consider functions that are linear on each polyhedral cell 
(or simplex in the simplicial case) and \emph{generic},
meaning that no two vertices have the same image under the function.
Such a theory was sketched in \cite{BK,Ku1} for obtaining lower bounds
for the number of vertices of combinatorial manifolds of certain type.

We now define genericity for finite abstract polytopal complexes
(for a definition see \cite[Ch.5]{Zie}). 
Examples are simplicial complexes and cubical complexes.
Moreover, any subcomplex of the boundary complex of a convex $d$-polytope is
a polytopal complex embedded in $\mathbb{E}^d$.

\begin{defi} 
Let $P$ be a finite (abstract) polytopal complex.
A function $f \colon P \to \mathbb{R}$
is called \emph{generic PL} if it is linear on each polytopal cell separately
and  if $f(v) \not = f(w)$ for any two distinct
vertices $v,w$ of $P$. As a consequence, $f$ is not constant on any
edge or higher-dimensional cell.

\smallskip
Similarly, if $P \subset \mathbb{E}^n$ is a compact polyhedron
with the structure of a polytopal complex, then any linear function 
on $\mathbb{E}^n$ induces a \emph{height function} on $P$. This height function $f$
is called \emph{generic} if the same condition is satisfied.
It is clear that for almost all directions in space (with respect to the 
Lebesgue measure) the associated height function is generic.

\end{defi}

We denote by $f_a$ and $f^a$ the sublevel set and the superlevel set:
$$f_a := \{x \ | \ f(x) \leq a \}, \ \ f^a := \{x \ | \ f(x) \geq a \}$$
\begin{lem} \label{deformation-retract}
If $f \colon P \to \mathbb{R}$ is generic PL and if 
$f^{-1}[a,b]$ contains no vertex of $P$, then $f_a$ is
a strong deformation retract of the sublevel $f_b$.
\end{lem}
\begin{proof} %Proof.
  If $P$ is a convex polytope then the assertion is obviously true.
Therefore it holds for any single cell of $P$ and -- in combination -- 
for the entire complex $P$.
\end{proof}

It is easy to construct
an isotopy that smoothly interpolates between
the level sets
% $f_{=a} := \{x \ | \ f(x) = a \}$ and
$f^{-1}(a)$ and
% $f_{=b} := \{x \ | \ f(x) = b \}$, resulting in
$f^{-1}(b)$, resulting in
% mappings between different level sets $f_{=t}$,
mappings between different level sets $f^{-1}(t)$,
%  $f_{=t'}$,
 $f^{-1}(t')$,
for $a\le t,t'\le b$,
that are piecewise linear.
With more technical effort one can construct such
an isotopy that is
piecewise linear even when considered as a function of all variables, including the interpolation parameter
$t\in[a,b]$
\cite[Section~4.2.3, Lemma 4.13 and Theorem 4.20]{Gr}.
We will make some more remarks about this topic in Section~\ref{isotopy}.
 
\medskip
 Lemma~\ref{deformation-retract} tells us that all points $p$ other than vertices
satisfy the regularity condition in Morse theory: The topology of
the sublevel does not change when passing through~$p$.
It remains to talk about the vertices since passing through a vertex
can definitely change the topology of the sublevel, as simple
examples show.
The topology can be measured preferably by topological invariants.
Therefore the following definition is suitable:
\begin{defi} \label{critical}
Let $f \colon P \to \mathbb{R}$ be generic PL and let $v$ be a
vertex with the level $f(v) = a$.
Then $v$ is called \emph{homologically critical for $f$} or \emph{H-critical} for short if 
$H_*(f_a, f_a \setminus \{v\};\mathbb{F}) \not = 0$
where $H_*$ denotes an appropriate homology theory with coefficients in a field~$\mathbb{F}$.
The total rank of $H_*(f_a, f_a \setminus \{v\})$ is
called the \emph{total multiplicity} of $v$ with respect to~$f$.
If 
\begin{displaymath}%\label{eq:index-critical}
H_k(f_a, f_a \setminus \{v\}) \not = 0 
\end{displaymath}
then we say that $v$ is
\emph{H-critical of index $k$}, and the rank of 
$H_k(f_a, f_a \setminus \{v\})$ is referred to as the corresponding 
\emph{multiplicity} of $v$ restricted to the index $k$. 
\end{defi}

{\sc Remark:} The idea behind this notion is that the homological type
of the sublevel set changes when passing through an H-critical point.
Since no two vertices have the same level under $f$, the homology
of $f_a \setminus \{v\}$ is the same as that for the open sublevel 
$(f_a)^\circ = \{x \ | \ f(x) < a \}$.

\medskip
By excision and the long exact sequence 
for the reduced homology $\widetilde{H}$
in a simplicial complex $P$ we can detect criticality 
in the link $lk(v)$ and the star $st(v)$ of a vertex $v$:
$$\widetilde{H}_k(f_a, f_a \setminus \{v\}) \cong \widetilde{H}_k(f_a \cap st(v), f_a \cap lk(v))
\cong \widetilde{H}_{k-1}(f_a \cap lk(v)) \cong \widetilde{H}_{k-1}(lk^-(v))$$
for $k \geq 1$ where $lk^-(v)$ denotes 
$$lk^-(v) := \{x \in lk(v) \ | \ f(x) \leq f(v)\} = lk(v) \cap f_a.$$ 
The homology of $lk^-(v)$ is the same as that of the full span of those vertices
in the link of $v$ whose level lies below $f(v)$.
Similarly we will use the notation
$$lk^+(v) := \{x \in lk(v) \ | \ f(x) \geq f(v)\} = lk(v) \cap f^a.$$ 

\medskip
This definition is also applicable to classical smooth Morse functions
on a smooth manifold. Then a critical point of index $k$ is also
critical with respect to Definition 3.3 with the same index, 
and the total multiplicity is always 1. 
Even for polyhedral surfaces the case of higher total
multiplicity occurs, as the example of a polyhedral monkey saddle shows.
It is easy to construct polyhedra such that there are critical vertices
of several indices simultaneously: Take the 1-point union of
a 1-sphere with a 2-sphere.

\medskip
{\sc Remark:} For polyhedra the homological definition used in \cite{CEH}
is equivalent to our definition above. It compares the homology of
the $(a-\epsilon)$-level with that of the $(a+\epsilon)$-level
if $a$ is the critical level.
However, for topological spaces in general both
definitions do not agree, as pointed out in \cite{Gove}.  
The problem with the incorrect \emph{Critical Value Lemma} in \cite{CEH}
is that a nested sequence of closed intervals can converge to a
common boundary point.
Then no open $\epsilon$-neighborhood around the critical level
can fit into any of the closed intervals.
Instead of the definition above one could compare the open sublevel 
$(f_a)^\circ = f_a \setminus f^{-1}(a)$ to the closed sublevel $f_a$.
For polytopal complexes (with closed polytopal faces)
this will lead to the same definition.

\medskip
There remains the possible case of $H_*(f_a, f_a \setminus \{v\}) = 0$
for some vertex $v$.
Since homology does not detect that it is critical we would like to call it 
\emph{non-critical} or \emph{regular}.
However, we have to be careful since regularity in the sense of
Lemma 3.2 is different. The question is: Can $f_{a+\epsilon}$ and $f_{a-\epsilon}$
be topologically distinct in this case?

\begin{defi}
A vertex $v$ with $f(v) = a$ is called \emph{homologically regular for $f$} 
or  \emph{H-regular} for short if
$H_*(f_a, f_a \setminus \{v\};\mathbb{F}) = 0$ for an arbitrary 
field $\mathbb{F}$.
\end{defi}
In classical Morse theory any H-regular point is actually regular
in a stronger sense (compare Section 4). 
We will see below that this is still true 
in dimensions $d \leq 4$ but it does not hold in general for PL manifolds and 
generic PL functions.

\begin{thm} {\rm (Morse relations, duality \cite{MC2,Kui1,Ku1})} \label{Morse duality}

Let $f \colon M \to \mathbb{R}$ be a generic PL function on a compact 
PL $d$-manifold $M$, and let $v_1, \ldots, v_n$ be the vertices.
By $a_i$ we denote the level $a_i = f(v_i)$. 
Then the \emph{Morse inequality}
\begin{equation}
  \label{Morse-inequality}
\sum_i{\rm rk} H_k(f_{a_i},f_{a_i} \setminus \{v_i\};\mathbb{F}) \geq {\rm rk} H_k(M;\mathbb{F})
\end{equation}
holds for any $k$ and any field $\mathbb{F}$.
Moreover,
\begin{equation}
\sum_k(-1)^k\sum_i{\rm rk} H_k(f_{a_i},f_{a_i} \setminus \{v_i\};\mathbb{F}) =  \sum_k (-1)^k {\rm rk}H_k(M,\mathbb{F}) = \chi(M).
\end{equation}
The expression ${\rm rk} H_k(f_{a_i},f_{a_i} \setminus \{v_i\};\mathbb{F})$
is nothing but the multiplicity of $v_i$ restricted to the index $k$,
and $\sum_i{\rm rk} H_k(f_{a_i},f_{a_i} \setminus \{v_i\};\mathbb{F})$
is the number $\mu_k(f)$ of critical points of index $k$, 
weighted by their multiplicities. Therefore the Morse inequality can  
also be written in the form
$$\mu_k(f)  \geq {\rm rk} H_k(M;\mathbb{F}).$$
%\medskip
Concerning the duality:

By Alexander duality in the link of a vertex $v$ one has
$\widetilde{H}_{d-k-1}(lk^+(v)) \cong \widetilde{H}_{k-1}(lk^-(v))$ 
for $1 \leq k \leq d-1$ and consequently
\begin{equation}
\widetilde{H}_{d-k}(f^a,f^a \setminus {v}) \cong \widetilde{H}_{k}(f_a,f_a \setminus {v}).
\end{equation} 
Clearly a local minimum of $f$ ($k=0$) is a local maximum ($k=d$)
for $-f$ and conversely.
This means that the number of critical points of $f$ of index $k$
coincides with the number of critical points of $-f$ of index $d-k$
(weighted with multiplicities).
\end{thm}

\begin{defi} {\rm (perfect functions, tight triangulations)}

  If a function $f$ satisfies
  the Morse
inequality
\eqref{Morse-inequality}
in Theorem~\ref{Morse duality}
with equality, for each~$k$,
then it is usually called a 
\emph{perfect function} or a \emph{tight function}. 
A \emph{tight triangulation} of a manifold is a triangulation 
such that any generic PL function $f$ 
%with $f(\sigma(1)) < f(\sigma(2)) < \cdots < f(\sigma(n))$ 
with arbitrarily chosen levels of the vertices is a tight function \cite{Ku2}.
\end{defi}

{\sc Examples:} A generic PL function $f$ 
on a compact surface without boundary
is perfect if and only $f_a$ is connected for any $a$.
On a simply connected compact 4-manifold without boundary it is
perfect if and only if $f_a$ is connected and simply connected for any $a$.
A triangulation of a surface is tight if and only if
it is 2-neighborly, one of a simply connected 4-manifold is tight
if and only if it is 3-neighborly. For any combinatorial sphere $K$ with $n$
vertices the power complex $2^K$ is a tightly embedded cubical manifold in 
$\mathbb{E}^{n+1}$, see \cite[3.24]{Ku2}.
\section{PL Morse functions}
\label{sec:our-critical-points}
By emphasizing the critical behavior of classical Morse functions 
(attaching a cell at each critical point) one
can adapt the classical Morse theory to the PL case as follows:
\begin{defi} \label{Morse function}
Let $M$ be a PL $d$-manifold and $f \colon M \to \mathbb{R}$ a 
generic PL function. 
\begin{itemize}
\item 
A point $p$ is called \emph{strongly regular} if there is a chart around $p$
such that the function $f$ can be used as one of the coordinates, i.e., if
in those coordinates
\begin{equation}
f(x_1, \ldots, x_d) = f(p) + x_d.
\end{equation}
If in a concrete polyhedral decomposition of $M$ distinct vertices
have distinct $f$-values, then $f$ is also generic PL, and moreover 
all points are strongly regular except possibly the vertices.

\item 
A vertex $v$ is called \emph{non-degenerate critical} if there is a PL chart
around $v$ such that in those coordinates $x_1, \ldots, x_d$ the function $f$
can be expressed as
\begin{equation}
f(x_1, \ldots, x_d) = f(v) - |x_1| - \cdots - |x_k| + |x_{k+1}| + \cdots + |x_d|.
\end{equation}
The number $k$ is then uniquely determined and coincides with the
index of $v$. The multiplicity is always $1$ in this case:
$H_k(f_a,f_a \setminus \{v\};\mathbb{F}) \cong \mathbb{F}$ and
$H_j(f_a,f_a \setminus \{v\}) = 0$ for any $j \not = k$.
The change by passing through the critical level can be either
$H_k(f_{a+\epsilon}) \cong H_k(f_{a-\epsilon}) \oplus \mathbb{F}$ 
or $H_{k-1}(f_{a-\epsilon}) \cong H_{k-1}(f_{a+\epsilon}) \oplus \mathbb{F}$. 
A function such that the second case never occurs is called a 
\emph{perfect function}. 
\item
The function $f$ is called a \emph{PL Morse function} if all vertices
are either non-degenerate critical or strongly regular.
In the terminology of \cite{Morse3} these are called 
\emph{topologically ordinary}
and \emph{topologically critical}, respectively. The function itself is
called \emph{topologically non-degenerate} in this case.
\end{itemize}
\end{defi}

The definitions of strongly regular and non-degenerate critical points
have in common that they require a local homeomorphism that transforms
$f$ into a certain PL~map~$g$.
It turns out that determining the topological type
of the embedding of $lk^-(v)$ into $lk(v)$ suffices
to verify such a requirement.
The connection between a characterization in terms of local charts
and equivalent characterizations in terms of $lk^-(v)$
is established by the following general fact:
There is a PL~homeomorphism between neighborhoods $N_v$ and $N_w$
mapping $v$ to $w$
and transforming a PL~map~$f$ on $N_v$ with $f(v)=0$
to a PL~map~$g$ with $g(w)=0$
if and only if
there is a PL~homeomorphism between $lk(v)$ and $lk(w)$
such that the signs of $f$ and $g$ at corresponding points agree.

For strongly regular points,
this observation leads to the following result:

\begin{lem} {\rm (strongly regular points)} \label{strongly regular}

        Let $f$ be a generic PL function on a combinatorial $d$-manifold.
	Then a vertex~$v$ with $f(v) = a$ is strongly regular 
        for $f$ if and only if 
	$lk^-(v)$ is a PL $(d-1)$-ball.
\end{lem}

In particular,
we obtain for strongly regular vertices~$v$
an embedding of a $(d-2)$-sphere into a $(d-1)$-sphere
that separates the latter into two $(d-1)$-balls,
namely, the boundary sphere $f^{-1}(a)\cap lk(v)$ of $lk^-(v)$
separates $lk(v)$
into the balls $lk^-(v)$ and $lk^+(v)$.
Such an embedding is called an unknotted $(d-1,d-2)$-sphere pair.
Thus, we can rephrase the previous characterization
in terms of unknotted sphere pairs:

\begin{cor}
	For dimension $d>1$,
	a vertex $v$ is strongly regular
	if and only if
	the pair $(lk(v),f^{-1}(a)\cap lk(v))$
	is an unknotted $(d-1,d-2)$-sphere pair.
\end{cor}

The question whether
all embeddings of $(d-2)$-spheres into $(d-1)$-spheres are unknotted
is the Schoenflies problem.
Since $f$ is generic,
the embedding of $f^{-1}(a)\cap lk(v)$ in $lk(v)$ is locally flat.
Therefore another characterization for strongly regular vertices
is possible for the cases
where the Schoenflies problem in the PL locally flat category
is known to have an affirmative answer.

\begin{cor}
	Let $v$ be a vertex of a combinatorial $d$-manifold~$M$
	with $d>1$ and $d\neq 5$.
	Then $v$ is strongly regular
	if and only if
	$f^{-1}(a)\cap lk(v)$ is a $(d-2)$-sphere.
\end{cor}

Similar considerations for non-degenerate critical points
yield the following characterizations:

\begin{lem}  {\rm (non-degenerate critical points)}

        Let $f$ be a generic PL~function
        on a combinatorial $d$-manifold.
        Then a vertex~$v$ is non-degenerate critical
        for $f$ with index~$k$
        if and only if
        $lk^-(v)$ is a regular neighborhood
        of an unknotted $(k-1)$-sphere embedded into
        the $(d-1)$-sphere~$lk(v)$.
\end{lem}
\begin{cor}
        Let $f$ be a generic PL~function
        on a combinatorial $d$-manifold.
        Assume that the vertex~$v$ is H-critical of index~$k$.
        Then $v$ is non-degenerate critical for $f$ with index~$k$
        if and only if
	the embedding of $f^{-1}(a)\cap lk(v)$ into $lk(v)$
        is PL-homeomorphic
        to the embedding of $S^{k-1} \times S^{d-k-1}$
        into the sphere~$S^{d-1}$
        given by the boundary of a regular neighborhood
        of an unknotted $S^{k-1}$ in $S^{d-1}$.
\end{cor}
Note that without the assumption of H-criticality,
the criterion still implies that $v$ is
non-degenerate critical with index~$k$ or index~$d-k$.

\begin{lem} {\rm (Morse Lemma)}

Let $f \colon M \to \mathbb{R}$ be a PL Morse function and assume
that there are no critical points with $f$-values in the interval $[a,b]$.
Then $f_a$ and $f_b$ are PL homeomorphic to each other, and 
$f^{-1}([a,b])$ is PL homeomorphic with the ``collar'' $f^{-1}(a) \times [a,b]$.

\end{lem}

\begin{cor} {\rm (Morse relations, duality)}

Let $f \colon M \to \mathbb{R}$ be a PL Morse function on a compact 
PL manifold $M$, and let $\mu_k(f)$
be the number of critical vertices of index $k$, then the \emph{Morse inequality}
\begin{equation}
\mu_k(f) \geq {\rm rk} H_k(M;\mathbb{F})
\end{equation}
holds for any $k$ and any field $\mathbb{F}$.
Moreover we have the \emph{Euler-Poincar\'{e} equation}
$$\sum_k(-1)^k\mu_k(f) = \chi(M)$$ 
and the \emph{duality}
$$\mu_{d-k}(f) = \mu_k(-f).$$
For a perfect function, $$\mu_k(f) = {\rm rk} H_k(M;\mathbb{F})$$
for all $k$.
This notion depends on the choice of $\mathbb{F}$.
\end{cor}
This follows from Theorem~\ref{Morse duality}.

\begin{cor} {\rm (Reeb theorem, \cite{Ko})}

Let $M$ be a compact PL $d$-manifold and $f \colon M \to \mathbb{R}$
be a PL Morse function with exactly two critical vertices.
Then $M$ is PL homeomorphic to the sphere $S^d$. 

\end{cor}
\begin{proof} %Proof.
  Since the minimum $p$ and maximum $q$ 
are always critical the assumption can
be reformulated by saying that any point between 
minimum and maximum is strongly regular.
Let us consider the restriction 
$$f_| \colon M \setminus \{p,q\} \to \mathbb{R}$$ 
without critical points.
%Any point has a coordinate neighborhood such that $f$ coincides
%with one of the coordinates.
For any level $f^{-1}(c)$ with $f(p) <  c < f(q)$ the Morse lemma tells us that
there is an $\epsilon > 0$ such that $f^{-1}(c-\epsilon,c+\epsilon)$
is PL homeomorphic with the cylinder $f^{-1}(c) \times (-\epsilon,\epsilon)$.
Furthermore there is a $\delta > 0$ such that 
$f^{-1}[f(p),f(p)+\delta]$ and $f^{-1}[f(q) - \delta,f(q)]$
are PL homeomorphic with $d$-balls.
Consequently $f^{-1}(f(p)+\delta)$ and $f^{-1}(f(p)-\delta)$
are PL homeomorphic with the $(d-1)$-sphere.
This implies that $f^{-1}[f(p)+\delta,f(q)-\delta]$
is PL homeomorphic with the cylinder 
$$f^{-1}(c) \times [p+\delta,q-\delta]
\cong S^{d-1} \times [p+\delta,q-\delta].$$
Putting the three parts together we see that 
$M$ is PL homeomorphic with the $d$-sphere~$S^d$.
\end{proof}%\hfill $\Box$

\medskip
{\sc Remark:} (a) In the smooth theory the same kind of proof leads only to
a homeomorphism to the standard $S^d$ but not to a diffeomorphism.
There are exotic 7-spheres admitting a Morse function with two
critical points, thus providing a counterexample.
By contrast it is well known that the $d$-sphere ($d \not = 4$) admits
a unique PL structure \cite[Thm.7]{Luft}. 
Therefore this problem could occur only for $d=4$.
But gluing together two standard 4-balls along their boundaries
leads to the standard 4-sphere. Therefore the proof above gives
a PL homeomorphy even for $d=4$.

(b) For the case of compact PL manifolds admitting a PL Morse function 
with exactly three critical points see \cite{EK}. 
The only possibilities occur in dimensions $d = 2,4,8,16$
with an intermediate critical point of index $k = 1,2,4,8$, respectively.

\medskip
{\sc Consequence:} (1) If there is an exotic PL 4-sphere then any PL Morse
function on it must have at least four critical points.

(2) If $M$ is a homology sphere that is not a sphere, then any PL Morse 
function $f$ on $M$ has at least six critical points. Consequently,
it cannot admit a perfect function.

\medskip
\begin{proof}[Proof of (2)] $M$ has a non-trivial fundamental group with a trivial
commutator factor group. Therefore $f$ must have a critical point of
index 1. This leads to a free fundamental group in the critical sublevel $f_a$.
If a critical point of index 2 introduces a relation in that group,
the quotient will be abelian. A non-abelian group requires a second generator,
and this requires a second critical point of index 1. Since the
fundamental group is not free, there must be a critical point of index
2 introducing a relation between the generators. 
By the Euler relation the number of
critical points must be even, so there are two critical points of index 1, 
minimum and maximum and two others.
\end{proof}

\medbreak
{\sc Example:} (3 critical points)
\nobreak

For the unique (and 3-neighborly and tight) 9-vertex triangulation 
of the complex 
projective plane \cite[Sect.~4B]{Ku2} any generic PL function 
assigning distinct levels to the
9 vertices is a PL Morse function with three critical points:
minimum, maximum and a saddle point of index 2 in between.
Since 123 is a 2-face  of the triangulation,
for the special case $f(1) < f(2) < f(3) < f(4) < \cdots < f(9)$
the sublevel $f_a$ will be a 4-ball for $f(1) < a < f(4)$ 
and the complement of a 4-ball for $f(4) < a < f(9)$. 
Since 1234 is not a 3-face of the triangulation, 
the critical sublevel $f_{f(4)}$
consists of the boundary of the tetrahedron spanned by 1234 extended
by sections through all 4-simplices except 56789. 

\medskip
{\sc Example:} (4 critical points)

There is a highly symmetric (and 3-neighborly and tight) 
13-vertex triangulation of the simply 
connected 5-manifold $M^5 = SU(3)/SO(3)$ \cite[Ex.5\_13\_3\_2]{Lu1}.
Any generic PL function assigning distinct values to the 13 vertices
will have total multiplicity 4, for special choices it will 
be a PL Morse function with minimum, maximum one saddle point 
of index 2 and one of index 3.
Since 135 is a 2-face of the triangulation, 
for a beginning sequence with $f(1) < f(3) < f(5) < f(7)$
any sublevel $f_a$ will be a 5-ball for $f(1) < a < f(7)$, 
the first critical level is $b = f(7)$ since 1357 is not a 3-face.
Again $f_b$ will be
the boundary of the tetrahedron 1357 extended by sections through 5-simplices.
According to $H_2(M^5;\mathbb{Z}) \cong \mathbb{Z}_2$ this empty tetrahedron 
1357 generates the second homology but twice the generator is homologous
to zero. 
Clearly 7 will be a saddle point for $f$ of index 2. 
However we extend this sequence,
by the Morse inequality $H_3(M^5;\mathbb{Z}_2) \cong \mathbb{Z}_2$ 
implies that there must be a critical point of index 3 also.

%\newpage

\section{Isotopy}
%\subsection*{Isotopy}
% % for stand-alone section, updating overview in introduction recommended
\label{isotopy}
We have mentioned after
Lemma~\ref{deformation-retract}
that successive level sets can be connected by an isotopy if there is no
vertex between them.
Such an isotopy can be used for visualization,
by putting some texture on the level sets in order to make it
% visually
clear how a level set moves as the level changes.

% It would be nice to establish the existence of a PL isotopy also for
% the case when
% the level set passes over a strongly regular vertex.
Fro{m} an application viewpoint, there are also quantitative aspects
that play a role here. One might look for isotopies that deform the level
sets as little as possible and that are PL while using few additional vertices.
Some results in this direction are given in~\cite[Section~6.2]{Gr}.

But already establishing the mere existence of a PL isotopy,
in particular for the case when
the level set passes over a strongly regular vertex,
is not a trivial matter.
As suggested in \cite{Rote},
such a PL isotopy can be represented by a PL homeomorphism
\[
\phi\colon f^{-1}(b) \times [a,b] \to f^{-1}[a,b]
\]
sucht that $f(\phi(x,t)) = t$ holds for all arguments.
We sketch an existence proof following \cite[Section 4.2.3]{Gr}.

If $f^{-1}[a,b]$ contains no vertices,
$f^{-1}(b) \times [a,b]$ and $f^{-1}[a,b]$
are combinatorially equivalent polytopal complexes.
% (with their natural structure as such complexes)
Triangulating these complexes by starring
at each vertex in corresponding orders
yields combinatorially equivalent simplicial complexes
and hence a PL homeomorphism by
simplexwise linear interpolation.
% (Since the combinatorial equivalence maps 
% vertices of $f^{-1}(b) \times \{a\}$ to vertices of $f^{-1}(a)$ and
% vertices of $f^{-1}(b) \times \{b\}$ to vertices of $f^{-1}(b)$
% the resulting map is level preserving)

It suffices to consider intervals $[a,b]$
such that $f^{-1}[a,b]$ contains a single regular vertex~$v$
with $f$-value $a$ or $b$.
% otherwise subdivide interval
Since the case $f(v) = a$ can be treated analogously,
we assume $f(v) = b$.

First, apply the isotopy construction for intervals without vertices
outlined above for $M \setminus (st(v))^\circ$, that is,
$M$ with the open star of $v$ removed.
This isotopy restricts to a PL homeomorphism %(name it for later use?)
fro{m} $(lk(v) \cap f^{-1}(b)) \times \{a\}$ to $lk(v) \cap f^{-1}(a)$.
Since $v$ is regular,
$(st(v) \cap f^{-1}(b)) \times \{a\}$ is a ball
bounded by the sphere $(lk(v) \cap f^{-1}(b)) \times \{a\}$
and $st(v) \cap f^{-1}(a)$ is a ball
bounded by the sphere $lk(v) \cap f^{-1}(a)$.
% (more precisely: 
% $st(v) \cap f^{-1}(b)$ is cone on 
% $lk(v) \cap f^{-1}(b)$ with apex $v$
% and $st(v) \cap f^{-1}(a)$
% is combinatorially equivalent to $lk^{-}(v)$)
The PL homeomorphism between the boundary spheres can be extended
to a PL homeomorphism between
the balls $(st(v) \cap f^{-1}(b)) \times \{a\}$ and $st(v) \cap f^{-1}(a)$.
This PL homeomorphism matches on $(lk(v) \cap f^{-1}(b)) \times \{a\}$
with the isotopy on the deletion of $v$.
Therefore we obtain a PL-homeomorphism
between
$(((M \setminus (st(v))^\circ) \cap f^{-1}(b))\times[a,b]) \cup (st(v) \cap f^{-1}(b)) \times \{a\}$
and
$((M \setminus (st(v))^\circ) \cap f^{-1}[a,b]) \cup (st(v) \cap f^{-1}(a))$
Now $(st(v) \cap f^{-1}(b))\times[a,b]$
can be considered as a cone on 
$((lk(v) \cap f^{-1}(b))\times[a,b]) \cup (st(v) \cap f^{-1}(b)) \times \{a\}$
with apex $(v,b)$,
and $st(v) \cap f^{-1}[a,b]$ as a cone on
$(lk(v) \cap f^{-1}[a,b]) \cup (st(v) \cap f^{-1}(a))$ with apex $v$.
Thus a cone construction defined by mapping $(v,b)$ to $v$ and 
interpolating between apices and bases
extends the given PL homeomorphism to a PL homeomorphism between
$f^{-1}(b) \times [a,b]$ and $f^{-1}[a,b]$
as desired.
% Level-preserving by construction,
% that is, holds at vertices and then by interpolation.

%\newpage
\section{Manifolds with boundary}
\label{sec:with-boundary}

The classical Morse theory was extended to smooth manifolds with boundary
$(M,\partial M)$ in \cite{Braess}. Here a \emph{Morse function} is defined
as a smooth function having only non-degenerate critical points in
$M \setminus \partial M$ and no critical points on $\partial M$, i.e.,
${\rm grad} f \not = 0$ on $\partial M$.
Furthermore the restriction $f|_{\partial M}$ is assumed to be a Morse
function on $\partial M$.

\begin{defi}\label{braess}
A critical point $p$ of $f|_{\partial M}$ is called \emph{$(+)$-critical for $f$}
if ${\rm grad} f|_p$ is an interior vector on $M$ (pointing into $M$). 
It is called \emph{$(-)$-critical for $f$}
if ${\rm grad} f|_p$ is an exterior vector on $M$ (pointing away from $M$). 
\end{defi}

\begin{pro} {\rm (Braess \cite{Braess})}\label{braess1}

Let $M$ be a compact smooth manifold with boundary, and 
let $\mu^+(f)$ and $\mu^-(f)$ denote the number of $(+)$- and $(-)$-critical
points. Only the $(+)$-critical points are H-critical and change the sublevel
by attaching a cell, the $(-)$-critical points are H-regular.
Moreover $f_{a-\epsilon}$ is a deformation retract of $f_{a+\epsilon}$ if
$f^{-1}[a-\epsilon,a+\epsilon]$ contains only a $(-)$-critical point on
$\partial M$ and no critical point in $M \setminus \partial M$. 
Then the Morse inequality reads as 
$$\mu(f|_{M \setminus \partial M}) + \mu^+(f) \geq {rk} H_*(M).$$
Moreover by duality on the boundary one has
$$\mu^+(f) + \mu^-(f) = \mu(f|_{\partial M}) \geq {\rm rk} H_*(\partial M).$$
However, there is no duality on $M$ since a point is $(+)$-critical
for $f$ if and only if it is $(-)$-critical for $-f$.
\end{pro}

For a proof see \cite[Satz 4.1 and Satz 7.1]{Braess}. In Satz 4.1 the 
assumption should be that the interval contains no critical point
in the interior and no $(+)$-critical point on the boundary.

\bigskip
In the case of a generic PL function we can directly apply
Definition~\ref{critical} with the following 
result for a vertex $v \in \partial M$
with $f(v) = a$ \cite{Ku0}:
$${\rm rk}H_*(f_a,f_a\setminus \{v\}) + {\rm rk}H_*(f^a,f^a\setminus \{v\}) 
\geq {\rm rk}H_*((f|_{\partial M})_a,(f|_{\partial M})_a \setminus \{v\})$$

{\sc Example:} Simple 2-dimensional examples show that the last inequality
is not always an equality:
It can happen that a boundary point is H-critical for $f$ but
H-regular for $f|_{\partial M}$. By integrating the number of critical
points over all directions of height functions we see that the
contribution of the boundary is half the integral over the boundary
separately in the smooth case and greater or equal to half this integral
in the PL case~\cite{Ku0}.

\bigskip
By combining the definitions for PL Morse functions in Section 4 with the 
ideas of Definition~\ref{braess} above we can formulate a theory of PL Morse
functions on manifolds with boundary as follows.

\begin{defi} \label{Morse function with boundary}
Let $M$ be a compact PL $d$-manifold with boundary 
and $f \colon M \to \mathbb{R}$ a generic PL function. 
Then $f$ is called a \emph{PL Morse function} if all interior vertices
are either non-degenerate critical or strongly regular
in the sense of Definition~\ref{Morse function}
and all vertices on $\partial M$ are either $(+)$-critical
or $(-)$-critical or strongly regular.

\smallskip
A point $p \in \partial M$ is called \emph{strongly regular} 
if there is a chart around $p$
such that $M$ is described by $x_1\le 0$
and the function $f$ can be used as the coordinates $x_d$
in $\partial M$, i.e., if in those coordinates
\begin{equation}
f(x_1, \ldots, x_d) = f(p) + x_d
\end{equation}
for $x_1\le 0$.
%and $x_2, x_3, \ldots, x_d$ is a chart around $p$ inside $\partial M$.
If in a concrete polyhedral decomposition of $M$ distinct vertices
have distinct $f$-values, then $f$ is also generic PL, and moreover 
all points are strongly regular except possibly the vertices.

\smallskip
A vertex $v\in \partial M$ is called \emph{non-degenerate $(+)$-critical}
(or \emph{$(-)$-critical}, respectively) 
if there is a PL chart with coordinates $x_1, \ldots, x_d$ 
around $v$
for which the set $M$ is described by the constraint
%such that in those coordinates $x_1, \ldots, x_d$ 
%we have
\begin{gather*}
x_d \geq - |x_1| - \cdots - |x_k| + |x_{k+1}| + \cdots + |x_{d-1}|
\\(\mbox{or } x_d \leq - |x_1| - \cdots - |x_k| + |x_{k+1}| + \cdots + |x_{d-1}|
\mbox{ respectively})  
\end{gather*}
 and the function $f$ can be expressed as
\begin{equation}
f(x_1, \ldots, x_d) = f(v) + x_d.
\end{equation}
See Figure~\ref{fig:boundary}
for an illustration.
In this case the boundary is represented by the equation
$$x_d = - |x_1| - \cdots - |x_k| + |x_{k+1}| + \cdots + |x_{d-1}|,$$
and the restriction $f|_{\partial M}$ is written as
\begin{equation}
f(x_1, \ldots, x_{d-1}) = f(v) - |x_1| - \cdots - |x_k| + |x_{k+1}| + \cdots + |x_{d-1}|,
\end{equation}
so $v$ is non-degenerate critical for $f|_{\partial M}$.
\end {defi}
\begin{figure}[htb]
  \centering
  \includegraphics{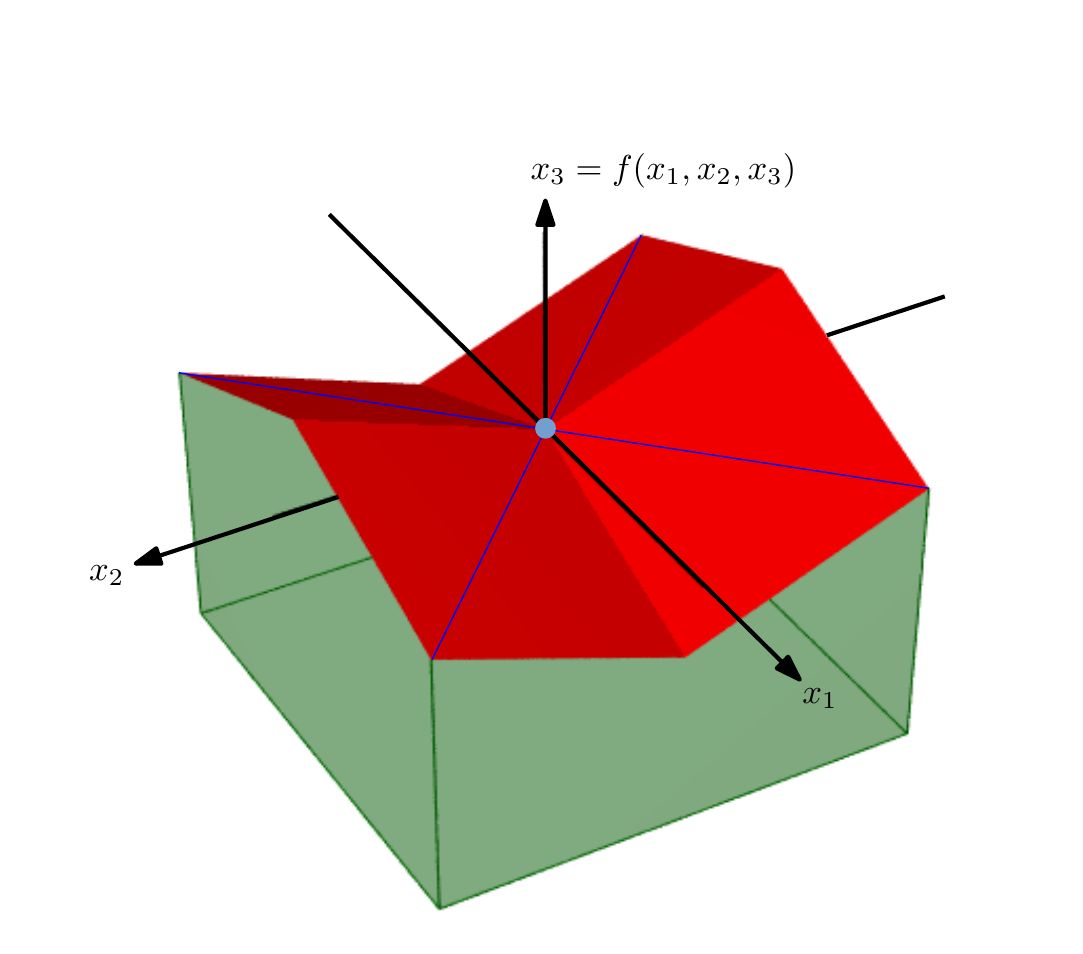}
  \caption{A non-degenerate critical point (blue) of index 1 on the boundary
    of a 3-manifold $M$. The boundary $\partial M$ is the corrugated
    red saddle surface.
    If $M$ consists of the volume unter the ``roof'', % $\partial M$,
    as indicated by the green ``walls'', then this is a $(-)$-critical point.
    If $M$ lies above the red surface, then it is a $(+)$-critical
    point.
  The blue cross is the level set at the critical value.}
  \label{fig:boundary}
\end{figure}
\begin{cor}
In the situation of Definition~\ref{Morse function with boundary}
only $(+)$-critical points on the boundary are $H$-critical,
necessarily with multiplicity 1 and index $k$. 
Any $(-)$-critical point on the boundary is $H$-regular.
\end{cor}
\begin{proof} %Proof.
  The number $k$ in Definition~\ref{Morse function with boundary}
is uniquely determined and coincides with the
index of $v$ if $v\in \partial M$ is $(+)$-critical, and the  
multiplicity is always $1$ in this case:
$H_k(f_a,f_a, \setminus \{v\};\mathbb{F}) \cong \mathbb{F}$ and
$H_j(f_a,f_a, \setminus \{v\}) = 0$ for any $j \not = k$.
The change by passing through the critical level can be either
$H_k(f_{a+\epsilon}) \cong H_k(f_{a-\epsilon}) \oplus \mathbb{F}$ 
or $H_{k-1}(f_{a-\epsilon}) \cong H_{k-1}(f_{a+\epsilon}) \oplus \mathbb{F}$. 
A function such that the second case never occurs is called a 
\emph{perfect function}.
For a $(-)$-critical vertex $v \in \partial M$ the homotopy types of
$f_a$ and $f_a \setminus \{v\}$ coincide.
\end{proof} %\hfill $\Box$

\begin{cor}
Proposition~\ref{braess1} remains valid for PL Morse functions
on PL manifolds with boundary.
\end{cor}

\section{Discrete Morse functions induce PL Morse functions}
\label{sec:Forman}

The above characterizations of strongly regular, non-degenerate,
$(+)$- and $(-)$-critical points also allow an easy proof for a construction
of PL Morse functions from discrete Morse functions.
For the connection between classical Morse theory and discrete Morse theory
see \cite {Be}. In particular for any smooth $d$-manifold with $d\leq 7$ the
set of smooth Morse vectors coincides with the set of discrete Morse vectors.

\begin{defi} {\rm (Forman \cite{Fo})}

 A \emph{discrete Morse function}
        maps cells of a complex to real numbers
        such that for each $k$-cell,
        there is at most one exceptional $(k-1)$-face
        whose value is not strictly smaller
        and at most one exceptional $(k+1)$-coface
        whose value is not strictly larger.
        A $k$-cell is called \emph{critical}
        if it has no exceptional $(k-1)$-face
        and no exceptional $(k+1)$-coface.

        Fact:
        No cell has
        both an exceptional face and an exceptional coface,
        hence pairing each non-critical cell
        with its exceptional face or coface
        yields a partial matching of immediate face/coface pairs.

        We call a discrete Morse function \emph{generic}
        if it has the following additional properties:
        The function is injective.
        Any non-immediate face of a cell has smaller value.

        Fact:
        Any discrete Morse function is equivalent to a generic one
        in the sense that
        it has the same critical cells and induces the same matching.

\end{defi}

\begin{lem} \label{discrete Morse}
	Any discrete Morse function on a combinatorial manifold~$M$
	induces a generic PL Morse function linear on cells
	of a derived subdivision of~$M$ such that
	non-critical cells correspond to strongly regular vertices
	and critical cells of dimension~$k$ correspond to
	non-degenerate vertices of index~$k$.
\end{lem}
\begin{proof}
	Let $K$ be the underlying complex of~$M$ and $g\colon K\to\mathbb{R}$
	a discrete Morse function, without loss of generality generic.
	Define $f$ on the domain of a derived subdivision of $K$
	by linearly interpolating the values at the vertices
	given by the assignment $f(v_S)=g(S)$
	for each cell $S\in K$ and its corresponding 
	vertex~$v_S$ in the derived.
	Observe that for a $k$-simplex~$S$ in $K$,
	the link of $v_S$ in a derived subdivision
	is the join of two spheres,
	namely the derived of $bd(S)$,
	formed by vertices corresponding to proper faces of $S$, 
	and a sphere formed by the vertices
	corresponding to proper cofaces of $S$.
	In particular, the embedding of the $(k-1)$-sphere
	formed by the derived of $bd(S)$ is unknotted
	in $lk(v_S)$.
	For a critical cell~$S$,
	this implies already the claim
	that $v_S$ is non-degenerate critical of index~$k$,
	because the subcomplex of $lk(v)$  spanned 
        by the vertices
        with $f$-value smaller than $g(S)$
        agrees with the derived of $bd(S)$
        in this case
        and hence $lk^-(v_S)$ is a regular neighborhood
        of an unknotted $(k-1)$-sphere.

%        $lk^-(v_S)$ is the derived of $bd(S)$ in this case
%	and hence $f_{g(S)}\cap lk(v_S)$ a regular neighborhood
%	of an unknotted $(k-1)$-sphere.

	For a non-critical cell~$S$ however, 
        the subcomplex of $lk(v)$
        spanned by the vertices
        with $f$-value smaller than $g(S)$
        is either
        the derived of $bd(S)$
        with the open star of a vertex $v_T$ removed,
        where $T$ is the exceptional face of $S$,
        or the join of the derived of $bd(S)$
        with a single vertex $v_{ST}$,
        where $ST$ is the exceptional coface of $S$.
        In any case, the subcomplex is a ball
        and its regular neighborhood
        $lk^-(v_S)$ is a ball as well,
        showing that $v_S$ is strongly regular.
\end{proof}

%        $lk^-(v_S)$ is either
%	the derived of $bd(S)$ with the open star of a vertex $v_T$ removed,
%	where $T$ is the exceptional face of $S$,
%	or the join of the derived of $bd(S)$ with a single vertex $v_{ST}$,
%	where $ST$ is the exceptional coface of $S$.
%	In any case, $lk^-(v_S)$ is a ball and its regular neighborhood
%	$f_{g(S)}\cap lk(v_S)$ is a ball as well,
%	showing that $v_S$ is strongly regular.

The construction from Lemma~\ref{discrete Morse}
also works for generic discrete Morse functions~$g$
on a combinatorial manifold~$M$ with boundary.
Then the boundary cells produce the following types of vertices
for the induced PL Morse function:
A critical boundary cell of dimension~$k$
corresponds to a $(+)$-critical point of index~$k$.
A non-critical cell
that is paired with a cell in the boundary,
i.e., the cell is also non-critical with respect to
the restriction of $g$ to the boundary of $M$,
corresponds to a strongly regular point.
A non-critical cell of dimension~$k$
that is paired with a cell not belonging to the boundary,
i.e., the cell is critical with respect to
the restriction of $g$ to the boundary of $M$,
corresponds to a $(-)$-critical point of index~$k$.

\section{The special case of low dimensions} 
\label{sec:low-dimensions}
Under the assumption that distinct vertices have distinct $f$-levels,
only vertices can be critical. The critical vertices
play the role of the critical points in classical Morse theory,
either in the version of non-degenerate points or -- more generally --
for generic PL functions where higher multiplicities are admitted.
However, the H-regular vertices that are not strongly regular
do not fit this analogy: They do not contribute to the
Morse inequalities and they have no analogue in the classical theory
since they do not allow the cylindrical decomposition in a neighborhood
with an isotopy between the upper and the lower sublevel.
In some sense they are the most exotic objects to be considered here.
Therefore the question is whether they can occur or not.
In low dimensions $d \leq 4$ this is indeed not the case.

\begin{pro}
A $1$-dimensional finite polyhedral complex is a graph. Any generic PL
function has only minima (index $0$) or critical vertices
of index $1$, possibly with higher multiplicity.
Any vertex which is H-regular for $f$ and for $-f$ simultaneously
is also strongly regular for both of them.

\medskip
For a $1$-dimensional manifold we have only minima (index $0$),
maxima (index $1$) and strongly regular vertices otherwise.
\end{pro}

\begin{proof} %Proof.
  Let $v$ be a vertex and $a = f(v)$.
The link of $v$ is a finite set of points, some below the $a$-level,
some above.
If $lk^-(v)$ is empty we have a local minimum, the total multiplicity is 1.
If $lk^-(v)$ consists of $r \geq 2$ points then $v$ is critical of index 1 with
the multiplicity $r-1$.
In the special case $r=1$ the point is H-regular.
For $-f$ we have to interchange $lk^-(v)$ and $lk^+(v)$.
If in addition $lk^+(v)$ consists of only one point then
$v$ is a vertex of valence 2 between one upper and one lower vertex.
Obviously $v$ is strongly regular in this case.
For a 1-manifold $lk(v)$ consists always of precisely two points,
so the condition follows from $r=1$ for one of the functions $f$ or $-f$.
\end{proof}% \hfill $\Box$

\begin{pro}\label{surfaces}
Let $M$ be a PL $2$-manifold (a surface) with a generic PL function 
$f \colon M \to \mathbb{R}$.
The critical points (vertices) are only of the following types:

1. Local minima (index $0$, multiplicity $1$),

2. local maxima (index $2$, multiplicity $1$),

3. saddle points (index $1$, multiplicity arbitrary).

Any H-regular vertex is also strongly regular, and any saddle point
is non-degenerate critical in the sense of Definition~\ref{Morse function} if
its (total) multiplicity is $1$ in the sense of Definition~\ref{critical}.

\end{pro}

A splitting process of saddle points with higher multiplicity into
ordinary saddle points is described in \cite[p.~93]{EHZ}.

\begin{cor}
Any generic PL function on a PL $2$-manifold is a PL Morse function if
the multiplicity of every saddle point is $1$.
\end{cor}
\begin{proof}[Proof of Proposition~\ref{surfaces}]
The link of a vertex $v$ is a closed circuit of edges.
If $lk^-(v)$ is empty we have a minimum, if $lk^-(v) = lk(v)$ we have
a maximum ($lk^+(v)$ is empty), in all other cases $lk^-(v)$ and
$lk^+(v)$ have the same number of components, say $r$ components.
Then $v$ is critical of index 1 and multiplicity $r-1$.
An ordinary (non-degenerate) saddle point has $r=2$, a monkey saddle $r=3$.

The case of a H-regular vertex corresponds to the case $r=1$.
Since $st(v)$ is a topological disc, this implies that  both
$st^-(v)$ and $st^+(v)$ are discs, fitting together
along the $a$-level which is an interval.
Then we can apply Lemma~\ref{strongly regular}.
%On this disc, the function $f$ is regular in the sense that the preimage
%of each value is an interval, possibly decomposed according to the
%combinatorics of $lk(v)$. This defines the PL chart around $v$ as required.

The case of an ordinary saddle point corresponds to the case $r=2$.
These two components in $lk^-(v)$ and $lk^+(v)$ determine one
coordinate line each such that the function $f$ is linearly decreasing
or increasing, respectively. The $f(v)$-level in between is the cross
of the two diagonals in that coordinate system. 
\end{proof}
% \hfill $\Box$

\begin{thm}
Let $M$ be a PL $3$-manifold with a generic PL function 
$f \colon M \to \mathbb{R}$.
The critical points (vertices) are only of the following types:

1. Local minima (index $0$, multiplicity $1$),

2. local maxima (index $3$, multiplicity $1$),

3. mixed saddle points (index $1$ or $2$ or both, multiplicity arbitrary).

Any H-regular vertex is also strongly regular, and any saddle point
is non-degenerate critical in the sense of Definition~\ref{Morse function} if
its (total) multiplicity is $1$.

\end{thm}
\begin{proof}
Let $v$ be a H-regular vertex (not a local minimum) 
with 
$$H_0(lk^-(v);\mathbb{F}) \cong \mathbb{F}, \quad 
H_1(lk^-(v)) = 0 \ \mbox{ and } \ H_2(lk^-(v)) = 0.$$
Therefore $lk^-(v) = f_a  \cap lk(v)$ is a subset of 
$lk(v) \cong S^2$ which is a homology point. 
This implies that it is a homotopy point also, hence contractible. 
Consequently, $lk^-(v) \subset S^2$  is a disc since it is also
a compact 2-manifold with boundary. Its complement is a disc also.
Then we can apply Lemma~\ref{strongly regular}.
%In the star of $v$ this defines a decomposition into two 3-balls such that
%the preimage of every value under $f$ is a polyhedral disc.
%This defines the PL chart around $v$ as required.  

Now let $v$ be a saddle point with total multiplicity 1.
This means that $lk^-(v)$ and $lk^+(v)$ are subsets of a 2-sphere
with homology of a 0-sphere and a 1-sphere, respectively (in any order).
So there are two discs in $lk^-(v)$ and a cylinder in $lk^+(v)$ 
or vice versa. Let us pick one point in each disc and a circle in the cylinder
as ``souls''. Then the cones from $v$ determine one coordinate direction
with decreasing $f$ and two directions with increasing $f$ (or vice versa).
This defines the chart according to Definition~\ref{Morse function}.
\end{proof} %\hfill $\Box$

\begin{thm} \label{dimension 4}
Let $M$ be a PL $4$-manifold with a generic PL function 
$f \colon M \to \mathbb{R}$.
Then any H-regular vertex is also strongly regular.
\end{thm}
\begin{proof}
Let $v$ be a H-regular vertex (not a local minimum)  with 
$$H_0(lk^-(v);\mathbb{F}) \cong \mathbb{F}, \quad 
H_1(lk^-(v)) = 0, \quad H_2(lk^-(v)) = 0 \ \mbox{ and } \ H_3(lk^-(v)) = 0$$
for any field $\mathbb{F}$.
Therefore $lk^-(v)$ is a subset of $lk(v) \cong S^3$ which is an
homology point for arbitrary $\mathbb{F}$, hence it is also
a homology point for $\mathbb{Z}$, in other words: it is $\mathbb{Z}$-acyclic.
The following argument is taken from \cite{LZ}:
$lk^-(v)$ is a compact 3-manifold  
which is $\mathbb{Z}$-acyclic,
so the Euler characteristic is $\chi (lk^-(v)) = 1$.
The Euler characteristic of the boundary is twice the Euler characteristic
of the entire manifold, so $\chi = 2$ for the boundary which therefore 
contains a 2-sphere as one connected component, 
tamely (or locally flat) embedded into a polyhedral $S^3$.
Then by the 3-dimensional Schoenflies theorem in PL \cite{Luft}
it bounds a 3-ball in $S^3$ on either side.
This in turn shows that in our case there is no other component of the
boundary since it would contradict the assumption that $lk^-(v)$ is acyclic.
Then we can apply Lemma~\ref{strongly regular}.
%Consequently, by the same argument as for 3-manifolds above
%$f$ is strongly regular.
\end{proof} %\hfill $\Box$
It is remarkable that embeddings of the dunce hat into the 3-sphere
cannot provide counterexamples since their regular 
neighborhoods must be 3-balls \cite{BeL}.

\bigskip
{\sc Remark:} In higher dimensions $d \geq 5$ one obstruction 
is that a homology point
contained in a vertex link is not necessarily a homotopy point, 
see Section 6 below.
In particular there are acyclic 2-complexes in the 4-sphere that
are not contractible \cite{LZ}, moreover there are particular embeddings of the 
contractible dunce hat into the 4-sphere with regular neighborhoods that
are again contractible but not 4-balls \cite{Zee}.
These phenomena make it impossible to carry over the proofs above
to dimensions higher than $d=4$.

\section{Counterexamples in higher dimensions} 
\label{sec:high-dimensions}
{\sc Example 1:} (Critical point of total multiplicity 1 containing a knot)

We start with an ordinary knot built up by edges in a combinatorial 3-sphere.
A concrete example is the 6-vertex trefoil knot in the 1-skeleton
of the Br\"uckner-Gr\"unbaum sphere with 8 vertices, see \cite[Fig.4]{Ku1}.
After barycentric subdivision the knot coincides with the full
subcomplex spanned by its vertices.
This combinatorial 3-sphere can be the link of a vertex $v$ in 
a 4-manifold.
Define a generic PL function $f$ with $f(v) = 0, f(x) < 0$
for all vertices $x$ on the knot, and $f(y) > 0$ for all the other
vertices $y$ in the 3-sphere.
This vertex $v$ will be critical for $f$ of index 2 and multiplicity 1,
so homologically it behaves like a non-degenerate critical point
of index 2 of a PL Morse function.
However, the critical level will
be a cone from $v$ to a knotted torus in $lk(v)$. Therefore $v$ is not
a non-degenerate critical point in the sense of Definition~\ref{Morse function}.

\medskip
{\sc Example 2:} (H-regular point that is not strongly regular)

There are homology spheres that are not homotopy spheres.
The most prominent example is the Poincar\'{e} sphere $\Sigma^3$
that can be defined as 
the quotient of the 3-sphere $S^3$ by the standard action of the
binary icosahedral group (this action can be visualized 
in the symmetry group of the 120-cell).
It admits a simplicial triangulation with only 16 vertices \cite{BjL}.
By removing an open 3-ball we obtain a space that is a homology point
but not a homotopy point since its fundamental group does not vanish.
By removing one open vertex star we find an example with 15 vertices
$v_1, \ldots, v_{15}$.
This simplicial complex $C$  can be embedded into a high dimensional 
combinatorial sphere $S^n_k$ with vertices $v_1, \ldots, v_k$, $k > 15$
such that $C$ is the full complex spanned
by those 15 vertices $v_1, \ldots, v_{15}$.
Then we can build a combinatorial $(n+1)$-manifold $M$ such that the
star of one vertex $v_0$ is this combinatorial sphere $S^n_k$.
The simplest example seems to be the suspension $S(S^n_k)$
of this combinatorial sphere $S^n_k$ with altogether $k+2$ vertices.
Next we define a simplexwise linear function $f$ on $M$ in such a way that
$$f(v_1) < f(v_2) < \cdots < f(v_{15}) < f(v_0) < f(v_{16}) < f(v_{17}) 
< \cdots < f(v_k)$$
and with arbitrary but distinct values for all the other vertices of $M$.
Then the vertex $v_0$ is H-regular for $f$ since
in the link below the level and above the level the homology is trivial.
However, it is not strongly regular since in the open vertex star
the sublevel of $v_0$ is not contractible and is therefore not an open ball.
In other words: Homology is unable to detect that $v_0$ is a 
non-regular point. It behaves exactly like any of the points in the
interior of a top-dimensional simplex (which of course is strongly regular).

\medskip
{\sc Example 3:} (H-regular point that is not strongly regular)

There is a $\mathbb{Z}$-acyclic but not contractible 2-dimensional 
simplicial complex $K$ with 23 vertices polyhedrally
embedded into a polyhedral 4-sphere \cite{LZ}.
This can be extended to a triangulation of the 4-sphere with additional vertices
outside $K$ such that $K$ coincides
with the full subcomplex spanned by the 23 original vertices.
As in Example~1 above one can define a generic PL function $f$ on some
PL 5-manifold such that
in the link of a vertex $v_0$ the sublevel is spanned by those 23 
vertices. Consequently $lk^-(v_0)$ is acyclic, so $v_0$ is 
H-regular for $f$. It is not strongly regular since $lk^-(v_0)$ 
is not contractible, so it cannot be a 4-ball and $f_a \cap st(v_0)$ cannot 
be a 5-ball.
%Consequently, $v_0$ is not strongly regular for $f$.

\smallskip
By further embedding of $K$ into higher dimensional spheres
it follows that a regular neighborhood of $K$ is always 
homologically trivial but not contractible. 
Consequently, for any $d \geq 5$ there is an example of a generic 
PL function on a PL $d$-manifold with a H-regular critical point that 
is not strongly regular. This bound is optimal by the results of Section 5.

\medskip
{\sc Example 4:} (Degenerate critical point of total multiplicity 1)

It is well known that the double suspension $S(S(\Sigma^3))$
of the Poincar\'{e} sphere $\Sigma^3$ in Example 2
is homeomorphic with the sphere $S^5$ (the so-called 
\emph{Edwards sphere} \cite{Lu1}).
However, since the link of certain edges is precisely $\Sigma^3$, 
the triangulation is not combinatorial and does not
induce a PL structure.
Nevertheless, we can define generic PL functions adapted to  
this 20-vertex triangulation of $S(S(\Sigma^3))$.
If this 5-sphere occurs as the link of a vertex $v$ in a 6-manifold,
then we can find a generic PL function such that $f(v) = 0$,
$f(x) < 0$ for all vertices of $\Sigma^3$ and $f(x) > 0$ for the others.
Then $v$ is a H-critical point that homologically behaves like a 
non-degenerate critical point
of index 4 and multiplicity 1 but it is degenerate, so $f$ will
not be a PL Morse function.

\section{A special obstruction: the dunce hat}
\label{sec:dunce-hat}
Homology is a weaker concept than homotopy. So one might conjecture
that a vertex $v$ is strongly regular whenever both $lk^-(v)$ and $lk^+(v)$
 are contractible, so that no homotopy group would detect anything 
critical (one might call this \emph{homotopically regular}). 
The results of Section 5 show that this is true for
generic PL functions on $k$-manifolds with $k \leq 4$.
Here we are going to show that this systematically fails to hold in
dimensions $k \geq 5$.

\medskip
The dunce hat is known to be a 2-dimensional space that
is contractible \cite{Zee}. Any triangulation of it is not collapsible
since there is no edge to start the collapse.
There are embeddings into the $k$-sphere for any $k \geq 3$ \cite{BeL}.
If such a triangulated dunce hat occurs as the spanning full subcomplex
of $lk^-(v)$ then neither
homology nor homotopy will detect that $v$ is a critical point.
However, $v$ will be strongly regular if and only if a regular neighborhood
of the embedded dunce hat is a $k$-ball.

\medskip
By the results of \cite{Ma,Zee} there exist embeddings of the dunce hat
into $S^4$ such that a tubular neighborhood is not a 4-ball but
Mazur's contractible 4-manifold with boundary. The boundary must 
be a homology 3-sphere. Here we present a simple model
based on an 8-vertex triangulation.
\begin{figure}[hbt]
  \centering
  \includegraphics{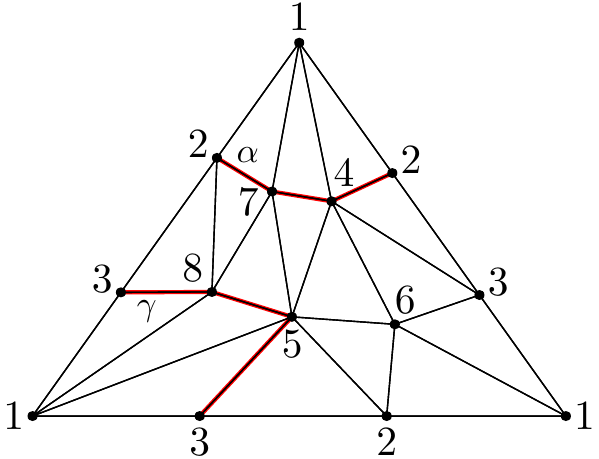}
\caption{A triangulated dunce hat, and two cycles $\alpha$ and
  $\gamma$ in the link of vertex~1.}
	\label{dunce}
\end{figure}
We start with the triangulation shown in Figure~\ref{dunce}.
It is equivalent to 
the triangulation used in \cite{BeL}.
Here is the list of triangles:
$$
124, 234, 346, 136,
126, 256, 235, 135,
127, 147, 278, 457, 578, 238, 138, 158,
456.
$$
It has the special property that any triangle contains either 1 or 8
or two vertices with consecutive labels $j,j+1$.
This implies that it can be embedded into the boundary complex of
the cyclic 5-polytope $C_5(8)$ with 8 vertices 1, 2, 3, \ldots, 7, 8
in that order.
Using Gale's evenness condition \cite{Zie},
we find the missing triangles: 246, 247, 257, 357.
The main question is: Is a tubular neighborhood of the 2-complex in
the 4-dimensional boundary complex of the cyclic 5-polytope a 4-ball or not?
It is certainly contractible since the dunce hat is.
One special property of the embedding is easily seen:
The two cycles $\alpha$ and $\gamma$ in \cite{Zee} are $(2472)$ and $(3583)$,
and these two are linked in the link of the vertex 1.
In fact, this is the cyclic 4-polytope $C_4(7)$ with 7 vertices, and that
contains the 7-vertex torus (see Figure~\ref{fig:torus}). The two cycles represent
$(1,1)$-knots on this torus, and any two of them are linked
like Hopf fibers.
Then \cite[Conjecture~3]{Zee} would imply that a tubular neighborhood
of the embedded dunce hat is not a 4-ball.
However, since we do not know whether this conjecture has been decided,
we
constructed a tubular neighborhood $M$, using the \textsc{Sage}\footnote
{\url{http://www.sagemath.org/}}
mathematics software
system, and
checked the fundamental group of its boundary~$\partial M$.
The fundamental group turned out to have a presentation
with two generators $u,v$ and the relations
$uvu^{-4}v = 1 = (v^2u^{-1}v^{-1}u^{-1})^2v$.
%$(a^{-1}b)^2(ab^{-1})^3aba^{-1}b = 1 = (a^{-1}ba^{-1})^3b^{-1}(ab^{-1}}a)^2b^{-1}.$
By introducing the extra relation $u^5 = 1$
we obtain $uv = (uv)^{-1} = v^{-1}u^{-1}$ and consequently
$$u^5 = v^7 = (uv)^2 = 1.$$ This group is known to be infinite
\cite[Sect.~5.3]{CM}. It coincides with the group of orientation preserving
automorphisms of the regular $(7,5)$-tessellation of the hyperbolic plane,
in accordance with \cite{Ma}.

As an independent confirmation, Benjamin Burton (private
communication)
analyzed $M$ with the
\textsc{Regina}
software for low-dimensional topology\footnote
{\url{https://regina-normal.github.io/}}.
\textsc{Regina}
 could
 simplify
 $\partial M$ %the boundary
 to
 9 tetrahedra, which it could recognize
in its built-in census database
 as
a Seifert fibred space, SFS [S2: (2,1) (5,1) (7,$-5$)].
In summary,
the result was in both cases that the boundary
 $\partial M$ of the tubular neighborhood
is not a 3-sphere.

\begin{cor}
A regular neighborhood of the $8$-vertex dunce hat above in the
boundary complex of the cyclic polytope $C_5(8)$ is a contractible
$4$-manifold with boundary but not a $4$-ball since its
boundary is not a sphere.
\end{cor}
\begin{cor} {\rm (explicit triangulation)}

The second barycentric subdivision of the cyclic polytope $C_5(8)$ contains
an explicit triangulation of a contractible $4$-manifold with boundary which
is not a $4$-ball.
\end{cor}
For the construction one just has to take the closed subcomplex
of all simplices that meet the embedded dunce hat in $C_5(8)$ above.
According to \cite{Be} this triangulation is not locally constructible.
\begin{cor}
There is a generic PL function on a $5$-manifold with a vertex $v$ that
is H-regular but not strongly regular and -- in addition --
with the special property
that both $lk^-(v)$ and $lk^+(v)$ are contractible. 
There are examples of this kind in every dimension $d \geq 6$
\cite{Ker}\rlap.\footnote
{see \url{https://en.wikipedia.org/wiki/Mazur_manifold}}
\end{cor}
For the construction we start with a combinatorial $5$-manifold
containing a vertex $v$ whose link is the boundary of the cyclic 
polytope $C_5(8)$; a concrete example is the cyclic polytope $C_6(9)$.
Then we define a generic PL function $f$ on the second
barycentric subdivision such that the open regular neighborhood of
the embedded dunce hat lies below $f(v)$ and its open complement lies above.
Then the level of $v$ itself in $lk(v)$ is a homology sphere but not a sphere,
in contrast with the characterization of Lemma~\ref{strongly regular}. 
\section{Computational aspects: Is regularity decidable?}
\label{sec:decidable}

The first problem is the \emph{manifold recognition problem}:
Given a pure simplicial complex of dimension $d$, can we
algorithmically decide whether it is the triangulation of
a combinatorial manifold?
More precisely, can we algorithmically decide whether all vertex links
are $(d-1)$-dimensional combinatorial spheres?
This is trivial for $d=1$ and fairly easy for $d=2$.
For $d=3$ we can decide whether a vertex link is a connected 2-manifold,
and then the Euler characteristic $\chi = 2$ is a sufficient criterion
for being a 2-sphere.
For $d=4$ we can first decide whether a certain vertex link is a connected
3-manifold. Then we can apply the sphere recognition algorithm
of A.~Mijatovi\'{c} \cite{Mij} and obtain:

\begin{cor}
It is algorithmically decidable whether a given simplicial complex
of dimension $d$ is a combinatorial $d$-manifold whenever $d \leq 4$. 
\end{cor}

\medskip
For a generic PL function on a PL manifold it is clearly decidable whether
a vertex $v$ is H-regular: One just has to compute
the integral homology of $lk^-(v)$. There are software packages to do so.
It is a much more delicate question to decide whether a vertex
$v$ is strongly regular. By the results of Section 5 
H-regularity is a sufficient criterion in low dimensions. 
Therefore we can state part (1) as follows:

\begin{cor}
(1) For a PL manifold $M$ of dimension $d \leq 4$ and a generic PL function $f$
on $M$ it is decidable whether
a particular vertex $v$ is strongly regular.

(2) Moreover, for $d \leq 4$ it is decidable whether a generic PL function
on $M$ is a PL Morse function or not.
\end{cor}

\begin{proof}[Proof of (2)]
By the results in Section 5 this is clear if $d \leq 3$. For $d=4$
we have to look at possible saddle points $v$ of index 1, 2 or 3
with total multiplicity 1. This can be decided by the homology.
In the case of index 1 $lk^-(v)$ consists of two homology points,
and $lk^+(v)$ consists of a homology 2-sphere, embedded into $lk(v) \cong S^3$.
By the argument used in Theorem~\ref{dimension 4} each homology point
is a 3-ball, and the homology 2-sphere is a regular neighborhood of an 
embedded 2-sphere. From this situation one can reconstruct a chart with
1 direction of decreasing $f$ and 3 directions with increasing $f$.
the case of index 3 is mirror symmetric to this situation 
(just interchange $-$ and $+$).
It remains to discuss the case of index 2 where both $lk^-(v)$ and $lk^+(v)$
 are homology 1-spheres that are linked in $lk(v) \cong S^3$.
But that means that on the critical level $f_a \cap f^a \cap lk(v)$
we have an embedded (connected) surface with $\chi = 0$, so it is a torus.
However, this torus can be knotted, see Example 1 in Section 6.
So in addition we have to decide whether this torus is unknotted.
This is known to be algorithmically decidable. 
If it is unknotted then it defines the chart according to 
Definition~\ref{Morse function}. 
If it is knotted then $f$ is not a PL Morse function.
\end{proof}%\hfill $\Box$

\bigskip
Concerning 5-manifolds we run into several problems: The Schoenflies problem 
is unsolved for embeddings of the 3-sphere into the 4-sphere, 
the Hauptvermutung is unknown for the 4-sphere, and an algorithm for
recognizing the 4-sphere (and hence: 5-manifolds) is not available.
(See however \cite{JLT} for practical approaches.)

\medskip
For $d$-manifolds of higher dimension $d \geq 6$, we even obtain
undecidability results.
Novikov proved 
\cite{VKF,CL,MTW} that
recognition of spheres in dimension~5 and above
is an undecidable problem.
In particular the manifold recognition problem is undecidable 
for $d$-manifolds with $d \geq 6$.

What are the consequences of Novikov's result for the recognition 
of strongly regular points?
Let us consider the suspension~$S(K')$
of an input~$K'$ for the sphere recognition problem
and define $f$ on $S(K')$ by choosing
a negative $f$-value for a single vertex~$w$ of~$K'$,
the $f$-value~$0$ for one vertex $v$ added by taking the suspension,
and distinct positive $f$-values for the remaining vertices.
If $K'$ is a sphere,
then this construction yields a strongly regular vertex~$v$,
because $lk^-(v)$ is a regular neighborhood
of the vertex~$w$ in $lk(v)=K'$,
hence a ball.
If $K'$ is not a sphere however,
not only the vertex~$v$ fails to be strongly regular,
its link~$K'$ witnesses that $S(K')$ fails to be a (closed) manifold as well.

This shows that the above construction yields a reduction
fro{m} the $d$-sphere recognition problem
to the recognition problem of strongly regular vertices
in arbitrary $(d+1)$-dimensional simplicial complexes.
Novikov's result renders the latter problem undecidable
for complexes of dimension at least~$6$.

\begin{pro}
	For arbitrary simplicial $d$-complexes with $d\geq6$,
	the problem of recognizing strongly regular vertices
	is undecidable.
\end{pro}

This reduction
 and its implied undecidability result
are somewhat unsatisfactory however.
The reduction produces manifold instances
only from positive instances of the sphere recognition problem,
whereas negative instances are reduced to non-manifold instances.
Hence the reduction establishes undecidability only if
verifying the manifold property
is considered to be part of the problem.
But, as noted above, recognizing $d$-manifolds for $d\geq6$
is already known to be undecidable in itself.

Therefore we would prefer a reduction
that produces manifold instances for the regular vertex recognition problem
fro{m} all instances of the sphere recognition problem.
For the proof of the following undecidability result,
we present a reduction that achieves this,
but at the cost of requiring higher dimension:
Instead of producing $(k+1)$-dimensional instances from $k$-dimensional ones,
it produces $2(k+1)$-dimensional instances.

\begin{pro}
	Recognizing strongly regular vertices of combinatorial $d$-manifolds
	with dimension~$d\geq 12$ is undecidable.
\end{pro}

\begin{proof}
	We sketch a reduction from Novikov's sphere recognition problem.
	The input instances for this undecidable problem are
	5-dimensional simplicial homology spheres,
	with positive instances being PL~spheres
	and negative instances having a non-trivial 
        fundamental group \cite[Theorem~3.1]{MTW}.

	Consider a simplicial complex~$K'$
	as input for Novikov's sphere recognition problem.
	Remove a maximal simplex from~$K'$.
	Embed the result as a subcomplex
	into the boundary sphere~$S'$
	of a 6-neighborly simplicial $d$-polytope
	for $d\geq 12$
	(more generally: a $(\dim(K')+1)$-neighborly simplicial $d$-polytope
	for $d\geq2(\dim(K')+1)$).
	Subdivide $S'$ to obtain an embedding as a full subcomplex.
	Denote the subdivided complex by~$S$
	and the full subcomplex representing $K'$ minus a simplex by~$K$.

	The suspension on~$S$
	is a combinatorial $d$-manifold, in fact, a $d$-sphere,
	with $S$ being the link of each of the two additional vertices.
	Define a function $f$ by choosing distinct values at the vertices
	such that one vertex~$v$ of the additional vertices has $f$-value~$0$,
	the vertices from $K$ have negative $f$-value,
	and the remaining vertices from $S$ have positive $f$-value.
	Then $lk^-(v)$ is a regular neighborhood of $K$ embedded 
        into $S$.

	If $K'$ is a sphere, then $K$ is a ball,
	and its regular neighborhood $lk^-(v)$ is a ball as well.
	Hence $v$ is a strongly regular vertex.
	On the other hand, if $K'$ has a non-trivial fundamental group,
	then, by the Seifert--van Kampen theorem,
	$K$ has the same non-trivial fundamental group.
	Since $K$ and $lk^-(v)$ are homotopy equivalent,
	the latter is not a ball,
	thus $v$ is not strongly regular.
\end{proof}

\subsection*{Acknowledgment}
This research was supported by the DFG Collaborative Research Center
TRR 109, `Discretization in Geometry and Dynamics'.
We thank Benjamin Burton for checking the topology of the tubular
neighborhood with the
\textsc{Regina} software.

%\newpage
{\small

\bigskip

Romain Grunert,
Institut f\"ur Informatik,
Freie Universit\"at Berlin, Takustr.~9, 14195 Berlin, Germany

E-mail: {\tt rgrunert@mi.fu-berlin.de}

\medskip
Wolfgang K\"uhnel,
Fachbereich Mathematik, 
Universit\"at Stuttgart, 
70550 Stuttgart, Germany

E-mail: {\tt kuehnel@mathematik.uni-stuttgart.de}

\medskip
G\"unter Rote,
Institut f\"ur Informatik,
Freie Universit\"at Berlin, Takustr.~9, 14195 Berlin, Germany

E-mail: {\tt rote@inf.fu-berlin.de}

}

\end{document}